\documentclass{amsart}
\usepackage{MnSymbol}

\usepackage{hyperref}
\hypersetup{
    colorlinks,
    citecolor=blue,
    filecolor=blue,
    linkcolor=blue,
    urlcolor=blue
}

\usepackage[all]{xy}
\usepackage{a4wide}
\usepackage{comment}

\DeclareFontFamily{OT1}{pzc}{}
\DeclareFontShape{OT1}{pzc}{m}{it}{<-> s * [1.100] pzcmi7t}{}
\DeclareMathAlphabet{\mathpzc}{OT1}{pzc}{m}{it}

    \usepackage{etoolbox}
    \patchcmd{\section}{\scshape}{\large\bfseries}{}{}
    \makeatletter
    \renewcommand{\@secnumfont}{\bfseries}
    \makeatother

\usepackage{tikz-cd}
\tikzcdset{arrow style=math font}
\tikzcdset{scale cd/.style={every label/.append style={scale=#1},
    cells={nodes={scale=#1}}}}

\usepackage{biblatex}
\usepackage{booktabs}
\bibliography{references}
\numberwithin{equation}{section}
\newtheorem{theorem}{Theorem}[section]
\newtheorem*{theorem*}{Theorem}
\newtheorem{corollary}[theorem]{Corollary}
\newtheorem{lemma}[theorem]{Lemma}
\newtheorem{proposition}[theorem]{Proposition}
\theoremstyle{definition}

\newtheorem{remark}[theorem]{Remark}
\newtheorem{example}{Example}

\def\DD{\mathpzc{D}}

\def\CC{\mathpzc{C}}
\def\MM{\mathpzc{M}}
\def\PP{\mathpzc{P}}
\def\US{\mathpzc{US}}
\def\UF{\mathpzc{UF}}
\def\hom{\mathsf{hom}}

\def\MM{\mathpzc{M}}
\def\zz{\mathbb{Z}}

\begin{document}

\title{Completions and Terminal Monads}
%started Jan 1, 2022
\author{Emmanuel Dror Farjoun}
\email{edfarjoun@gmail.com}
% started early 2022
%\email[Emmanuel Dror Farjoun]{farjoun at math.huji.ac.il}
%started Dec 2021
%{\today}
%\maketitle
%\begin{abstract}
\author{Sergei O. Ivanov} 
\email{ivanov.s.o.1986@bimsa.cn, ivanov.s.o.1986@gmail.com}
\thanks{The second named author is supported by BIMSA}

\date{\today}

\begin{abstract}
    We  consider  {\em the terminal monad} among those preserving the objects of  a subcategory  $\DD\subseteq \CC,$ and in particular preserving the image of a monad over the category   $\CC.$

    Several common monads $\CC\to \CC$ 
are  shown to be  uniquely characterized by the property of being terminal objects in the category of co-augmented endo-functors.  Once extended to infinity
    categories, this gives, for example, a complete characterization of the well-known Bousfield-Kan
    $R$-homology completion $R_\infty.$ In addition, we note that an idempotent pro-completion tower  $M_\bullet$
    can be associated with any co-augmented endo functor $M,$ whose limit $M_\infty$ is the 
    terminal monad that preserves the closure of  $Im M,$ the image  of $M,$ under finite limits.
    We conclude that  some basic properties of the homological completion tower $R_\bullet X$ of a space
    can be formulated and proved for  general monads over any category with limits, and characterized as universal.
  \end{abstract}
\maketitle

\section{Introduction and main results}

Many well-known and extremely useful constructions, mostly known as "completions", such as the profinite completion $G\to \widehat G$ or the Bousfield-Kan homology completion $X\to R_\infty X,$  are  usually  constructed directly, without  specifying what universal property, if any,  determines them up to equivalence.
Here, using a notion that we call "terminal monad," many of these are shown to be completely determined by the property  of being {\em terminal objects} in an appropriate  category of co-augmented functors $X\to F(X)$ over the given underlying category $\CC.$

Our first observation is  that  
 the category $$Fun(\CC,\CC)_{Id/-}$$ of co-augmented functors over a
category $\CC$ with limits, is itself closed under limits. Moreover,  consider any collection $D$ of objects in $\CC;$ and denote  by ${\mathcal F}ix_D(\CC),$ the subcategory  of the above functor category, consisting of functors that preserve each object of $D,$ namely with $d\to F(d)$ an equivalence for every $d\in D.$ Then 
this category, ${\mathcal F}ix_D(\CC),$  of 
co-augmented 
functors,  is also  closed under limits. In particular,  it has  a terminal object $Id\to M_D,$ which is easily shown to be a monad over $\CC.$

For a given small full subcategory $\DD\subseteq \CC,$ a  {\em construction} of the terminal monad, associated with the set of its objects,  can be  done by  re-considering the well-known   co-density monad, denoted here by $T_\DD:\CC\to \CC,$ associated with  a full subcategory  $ \DD\subseteq \CC,$ of a  category  $\CC$ which is  always assumed to be closed under limits. 

Hence, as above,  this functor $T_\DD$  can be characterized as {\em the terminal monad} on $\CC,$ among all co-augmented  functors $Id\to F$ on $\CC,$  that "preserves the objects of $\DD,$" i.e. with  $d\cong F(d)$ for all $d\in \DD.$  This  $d$ is "fixpoint" in the terminology of Adamek,  \cite{Adamekcolimits} definition 2.5, see also \cite{fixedpointnlab}.  It turns out that a terminal monad can be associated with more general,
not necessarily fully faithful, functors. In particular, we consider  terminal monads associated with a given  monad, or more generally with co-augmented endo-functors $M:\CC\to \CC.$  Notice, as in the references above, that the category of monads
over $\CC$ is also closed under all limits (---but,  in general, not under colimits.)

This allows one to characterize, by a universal property,  common constructions,  such as "completions," as  {\em terminal monads}
with respect to an appropriate subcategory.  There is an infinity-categorical extension of this 
observation \cite{Regimovmonads}. It leads, for example, to a seemingly new
characterization by a universal property of the well-known 
Bousfield-Kan homological completion $R_\infty$   as an infinity monad on topological spaces or  a simplicial sets $X.$  Compare \cite{BousfieldKan1973completion}: The completion $R_\infty$  is  shown to be {\em the terminal 
$\infty-$monad associated with the monad
$X\to R(X),$} among all co-augmented functors that preserve, up to homotopy, the essential image of $R,$
the free $R-$ module functor.  
Similar characterization of the pro-finite completion of a group, an algebraic variety, or a topological space, and other "completion" are  examples. In addition, following 
and elaborating on 
 Fakir \cite{Fakiridempotentmonads}, and Casacuberta-Frei \cite{CasacubertaFrei}, one can associate to a monad (or any co-augmented functor) a terminal {\em idempotent}  monad, i.e. a terminal localization functor, $L_M,$ projecting $\CC$ to the smallest subcategory of $\CC$ that contains the image of $M$ and is closed under limits. 
 The above definitions and constructions can be evidently dualized
to get analog ones associated with a co-monad
$N\to Id.$ In \cite{Yanovskimonadictower},
L. Yanovski constructs for  quite general $\infty-$categories a transfinite tower of co-monads 
with similar (implicit) properties, and strong transfinite convergence results.

$$$$

\subsection{Examples:}  To begin, consider some quite well-known examples. Recall that any localization (or so-called reflection)functor $L:\CC \to \CC$ is a terminal monad see \cite{CasacubertaFrei}. It is the terminal that preserves all the $L-$ local objects.
%Next, when $\DD\subset \CC$  is the subcategory of finite groups in the category of groups, then $T_{\DD}$ is the ({\em discrete!}) profinite completion functor on groups. 
Next,  the canonical set $U(X)$ of all ultra-filters on a set $X$
is a special case,  see \cite{TomLeinster}. Namely,  the monad $U$ now appears as  the terminal monad  among all co-augmented functors on sets that preserve every finite set. Another well-known example is the double dual of a vector space functor $V\to V^{**}.$ It is the terminal monad that preserves the one-dimensional spaces, or alternatively all finite-dimensional spaces. The last examples are clearly related to theorems \ref{introdoubleduald} and
\ref{introdoubledualdn} below.

It turns out, see below,  that even when  $\DD$ is  just the subcategory  spanned by  a three-element set $3=d\in\CC=Sets$ in the category of sets, then
$T_\DD=T_3(X)$ is again the  {\em underlying set } of the Stone-\v{C}ech compactification of the set $X,$ i.e.  $U(X)$ as above. Further, $T_2$ is a canonical sub-monad of the ultrafilter monad, while
$T_n$ for $n\geq 3$ is again the ultrafilter monad.
When $\DD\subseteq\CC$ is the subcategory of finite groups in the category of (discrete)
groups, then the {\em discrete} profinite completion  endo-functor, on the category of groups appears as the terminal monad among all co-augmented ones   $F,$ that preserve  all finite groups i.e.
with $\Gamma\cong F(\Gamma)$ for every finite $\Gamma.$
Or again, if $\DD\subseteq\CC$  is the subcategory of nilpotent groups in the category of groups, the associated monad $T_{\DD}$
is the (discrete) nilpotent completion functor in the category of groups.  For a ring $A,$ the (discrete) completion  functor of an  $A-$module, with respect to an ideal $I\subseteq A,$ can be similarly expressed as a terminal monad. A final, slightly stretched example, in an $\infty$-category, is the double dual as in equation \ref{eq:enddual} section \ref{operadiccompletion}  below, and Theorem \ref{operad}. This is very close to
  Mandell's functor, \cite{Mandellpadichomotopy}, see remark \ref{remarkmandel} below,
  that can be considered as the terminal monad preserving
  certain GEM spaces expressed as a double dual monad.

\subsubsection{Acknowlegements} This line of thought was a result of a private discussion with  M. Hopkins about the properties of the Bousfield-Kan
$R$-completion. Our students Guy Kapon and Shauly Regimov took
an active part in the discussion leading to the present paper. Their work led them to the corresponding formulations in the context of $\infty$-algebras, see \cite{Regimovmonads}.

\subsection{A sample of results}
The results  below regard the existence, basic properties, and explicit formulas for the terminal monad in certain cases.
Our first concern is to guarantee the existence of terminal monads
under certain rather weak conditions.
In any category, a terminal object can be considered as the limit over the empty diagram.
Hence the existence of a terminal 
co-augmented functor in a given 
functor category would follow from its closure under limits.

In  the following, the closure under limits and thus  the existence of a terminal object 
is guaranteed by the closure of the basic category 
 $\CC $ under limits.  In the present case, the functor categories, coma categories, and considered subcategories  are clearly closed under limits.  Limits 
 in the category of co-augmented  endo-functors are taken in the appropriate coma category under the identity functor.

%Terminal monads are the main objects of interest in the present work, beyond their existence we consider explicit 
%constructions, properties and examples.

The following gives a general  construction of the terminal monad in quite a general  framework, see Proposition \ref{D-completion} below.

$$$$
\begin{proposition} Let $\DD\subseteq \CC$ be a full subcategory of  a  category with limits. The co-density functor, or the  $\DD$-completion, $T_\DD:\CC\to \CC$  is the terminal object in the category of co-augmented functors $Id\to F\in \CC^{\CC},$ such that  the co-augmentation 
 map $d\to F(d)$ is an isomorphism for all $d\in \DD.$
This functor has a unique canonical monad structure.
\end{proposition}

\medskip

The following statements use the notation of
\ref{operad}, so they should be read with caution:
Although not treated here, they hold also in a complete monoidal category with internal
$hom(-,-)$ objects. In all cases, the  notation
$hom_O(-,-)$ should be read as the appropriate
equivariant maps, with respect to the implied action on the monoid $End$ or operad $O,$ on the range and domain. 

\begin{theorem}(See equation  \ref{eq:enddual})\label{introdoubleduald}
Let $d\in \CC$ be an element in a  category with limits. Denote by
$ End(d)$ the full subcategory generate by $d,$ namely the endomorphism of $d.$ The terminal monad $Id\to T_d$ that preserves $d,$
is given by the "structured double dual"  with respect to $d:$

$$X\to  T_d(x)=hom_{End\,(d)}(hom (X,d),d) $$
\end{theorem}

Further, the terminal monad that preserves an element $d\in \CC$ and all its cartesian powers $d^n,$
is given by a similar expression as below, where $\sf O={\sf O}_d$ denote the full endomorphism operad
of an object $d,$ given by all the morphisms $d^i\to d$ with $i>0.$

In the notation of \ref{introdoubleduald} one has:

\begin{theorem}\label{introdoubledualdn}

The terminal monad $Id\to T_{d^\bullet}$ that preserves $d^i,$ for all $i,$
is given by the "operadic double dual"  with respect to $d:$

$$X\to  T_d(X)=hom_{\sf O}(hom (X,d),d)
$$
\end{theorem}

$$$$

\subsubsection*{Terminal monads associated to a given monad $M$}
For a given monad $Id\to M:\CC\to \CC,$ ( or, more generally, a co-augmented endo-functor,) one has an associated {\em terminal monad } $T_M:\CC\to \CC,$ which is the terminal endo-functor
among all those that preserve the image of $M,$
namely with $M\to FM$ an isomorphism.
As an  example,  of such a monad $M$ one can take
any of the monads discussed above or even the terminal monad $T_\DD$  as above.  The terminal monad associated with the (discrete) profinite completion  functor, $G\to M(G)= \widehat G\equiv proG$ is a functor $T_{pro},$ that preserves all groups of the form $\widehat G,$  i.e. groups that are the {\em discrete} profinite completion of  some group $G.$  Next, if $U$ is the ultrafilter monad discussed above then its associated terminal monad can be seen to be the identity
monad, $T_U=Id,$  which is clearly the only monad that preserves all possible sets of the form $U(X)$, since the latter have arbitrarily high cardinality.
$$$$
The terminal monad  associated to $M$ can be expressed explicitly as follows: Compare  
  \cite{Fakiridempotentmonads}:

\begin{theorem} 
Let $M$ be a monad on a  category $\CC.$ The  associated terminal monad 
is given as the equalizer

$$
T_M \cong Equal (M\rightrightarrows M^2).
$$

\end{theorem}

In addition, we may consider, following Fakir above, the category of  idempotent monads, which are often called localizations or reflections.  Fakir constructs  for every monad  $M$ a naturally associated  idempotent monad K(M). Casacuberta et el 
observed  in \cite{CasacubertaFrei} that this idempotent monad is  {\em terminal} among all idempotent monads   $F, F\cong F^2,$ with the property  $M(f)$ is an isomorphism if and only if $K(M)(f)$
is one. Note the difference between $T_M$
and $K(M).$ The latter is discussed shortly in the last section below. The $\infty$-category analog is clearly the totalization of the co-simplicial
monad $M^\bullet$ discussed in \cite{Regimovmonads}.

$$$$
\subsection{Outline of the rest of the paper}

We begin with recalling  the general concept of completion i.e. co-density  with respect to a subcategory, such as the subcategory of compact objects.
This is done by considering the right Kan extension of a subcategory over itself. This  gives many   known examples of terminal monads. We then consider the terminal monad associated with a given object in a category and one associated with a given monad. The last example gives a functor from monads to terminal monads on the category $\CC.$  The paper goes on to consider some known special cases such as the category of sets and groups where the general construction gives some well-known constructions as a terminal monad, this
characterizes them uniquely by a property. The last section deals with the pro-idempotent monad associated with a co-augmented endo-functor, vastly generalizing the classical Bousfield-Kan completion tower $R_\bullet,$ here  only for a discrete category,
but paving the way for a
similar result for an $\infty$-category.

\section{ \texorpdfstring{$\DD$}{}-Completions}

Let $\DD$ be a full subcategory of a category $\CC$. Denote by $I:\DD\to \CC$ the embedding and assume that the right Kan extension of $I$ by $I$ exists and denote it by 
\begin{equation}
   T= T_\DD={\sf Ran}_I(I) : \CC\to \CC.
\end{equation}
So by the definition of the right Kan extension, $T$ is a functor together with a natural transformation
$\varepsilon: T I  \longrightarrow I$

\begin{equation}
\begin{tikzcd}
& \CC  \ar[rd,"T"] \ar[d,Rightarrow,"\varepsilon"]
 & \\
\DD\ar[ru,"I"] \ar[rr,"I"']  & \  & \CC \end{tikzcd}
\end{equation}

such that for any functor $F:\CC\to \CC$ and any natural transformation $\varepsilon':FI \to I $ there exists a unique natural transformation $ \delta: F \to T $ such that $ \varepsilon \circ  (\delta I) = \varepsilon',$ where $\delta I: FI \to T I$ is the whiskering of $\delta$ and $I$. The equation $ \varepsilon \circ  (\delta I) = \varepsilon'$ can be rewritten as follows: for any $d\in \DD$ 
\begin{equation}
\varepsilon_d \circ \delta_d = \varepsilon'_d.    
\end{equation}
Note that the universal property implies that for two natural transformations $\delta,\delta':F\to T$ the equation $\delta I=\delta' I$ implies $\delta=\delta'.$ Thus $T$ appears already here as
a terminal functor, in a somewhat  different sense from the above. Compare \cite{Ivandeliberticodensity}.

The functor $T$ will be called the functor of $\DD$-completion. 
Any right Kan extension can be presented as a limit over a comma category \cite[Ch. X, \S 3, Th.1]{mac2013categories}. In our case, it is just the limit of the projection functor
\begin{equation}\label{limitovercat}
    T(c) = {\sf lim} (c \downarrow \DD \to \DD).
\end{equation}

\begin{lemma}
The morphism $\varepsilon: TI\to I$ is an isomorphism. In particular, for any $d\in {\sf Ob}(\DD)$ we have an isomorphism
\begin{equation}
\varepsilon_d:T(d) \cong d.
\end{equation}
\end{lemma}
\begin{proof}
Since the functor $I:\DD\to \CC$ is full and faithful, by \cite[Ch. X, \S 3, Cor.3]{mac2013categories} we obtain that $\varepsilon$ is an isomorphism. 
\end{proof}

\begin{lemma}\label{lemma:eta}
There exists a unique natural transformation 
\begin{equation}
\eta^T: {\sf Id}_\CC \longrightarrow T
\end{equation}
such that $\eta^T_d=\varepsilon_d^{-1}$ for any $d\in {\sf Ob}(\DD).$ 
\end{lemma}
\begin{proof}
Take $F={\sf Id}_\CC$ and $\varepsilon'={\sf id}_I $ and use the universal property of the right Kan extension.
\end{proof}

Further we will treat $T$ as an co-augmented functor $T=(T,\eta^T).$ For any co-augmented functor $F=(F,\eta^F)$ we set 
\begin{equation}
    {\sf Inv}(F)=\{c\in {\sf Ob}(\CC)\mid \eta^F_c \text{ is iso} \}.
\end{equation} 
Note that $\DD \subseteq {\sf Inv}(T).$

\begin{lemma}\label{lemma:retract}
The class ${\sf Inv}(F)$ is closed under retracts.
\end{lemma}
\begin{proof}
Because a retract of an isomorphism is an isomorphism. 
\end{proof}

The following is a basic observation that follows from the above, that justifies the term "terminal monad," compare [3.7.3] in \cite{FrancisBorceuxcategorical}.

\begin{proposition}\label{D-completion} The co-augmented functor $T_{\DD}:\CC\to \CC$ of the $\DD$-completion is a terminal object in the category of co-augmented functors $F$ for which $\DD\subseteq {\sf Inv}(F).$ Moreover, $T_{\DD}$ is a monad  in functor category  $\CC^{\CC}.$  
\end{proposition}

\begin{proof} Take a co-augmented functor $F=(F,\eta^F)$ such that $\DD\subseteq {\sf Inv}(F).$  Note that $\eta^F I:I\to FI$ is an isomorphism. 
We use the universal property of $T$ and take $\varepsilon'=(\eta^F I)^{-1}:FI\to  I.$ Then there exists a unique natural transformation $\delta : F \to T$ such that $ \varepsilon \circ (\delta I)=(\eta^F I)^{-1}.$ Since $ \varepsilon = (\eta^F I)^{-1},$ we obtain that the equation $\varepsilon \circ (\delta I)=(\eta I)^{-1}$ is equivalent to the equation $ (\delta \circ\eta^F)I = \eta^T I.$ And the equation $(\delta \circ\eta^F)I = \eta^T I$ is equivalent to $ \delta \circ\eta^F=\eta^T$ by the universal property of $T.$ 

The monad structure of this terminal $T_\DD$ follows immediately from the fact that its square preserves all objects of $\DD,$ giving a unique natural transformation $T_{\DD}^2\to T_\DD.$  The monadic equations are satisfied since they all involve equality of natural transformations  from powers of $T_\DD$ to $T_\DD$  itself, but there is a unique such transformation for each power since 
$T_\DD$ is terminal among  these co-augmented functors, all of which preserve the objects of  $\DD.$ 
\end{proof}

{\em Remark:} Note that the above characterization shows that the terminal monad $T_{\DD},$ associated with a subcategory
$\DD\subseteq \CC,$ can  be identified using  solely  its effect on the objects in $\DD,$ being the terminal
co-augmented  functor  $\CC \to\CC,$ that "preserves the objects" of this subcategory.
Of course, its usual construction, as above, does  employ  morphisms  in $\DD$  and $\CC.$

\section{Terminal monads associated with a  functor \texorpdfstring{$\DD\to \CC$}{}} 
More generally, consider a general functor $F:\DD\to \CC.$ 
Now, consider the  subcategory  
$\CC^{\CC}_F,$ of the category of end-functors  $\CC^{\CC},$ consisting of all co-augmented 
functors  G, $ Id\to G:\CC\to \CC,$ 
that preserve the image of $F,$
i.e. with 

$$  (Id\to G)(F(x))= F(x)\to GF(x)$$ 

is an isomorphism  in $\CC$ for any object $x\in \DD.$  
The  subcategory    $\CC^{\CC}_F,$ of the full functor category,  is a  category of co-augmented functors $G,$  which is evidently closed under all limits. Hence it has a terminal object which is the  {\em terminal monad}
$T_F$ associated with the given functor $F,$ namely, preserving the image of $F.$

We note that this terminal object has a natural monad structure:

\begin{proposition} Let $M$ be an co-augmented endo-functor in $\CC^{\CC.}$
 The terminal object, $T_M,$ in the category $\CC^{\CC}_M$ of co-augmented endo-functors preserving the image of $M,$ is naturally a monad.   
\end{proposition}
\begin{proof}
Denote the terminal co-augmented functor by $T_M$ as above. Since the composition: $T_M\circ T_M$ clearly preserves the image of $M,$ and $T_M$ is terminal among those preserving $M,$ there is a unique map $\mu: T_MT_M\to T_M.$  The  conditions,  on a co-augmented functor with this  $\mu$ as a structure map, of being  a monad, involve  equality among various  maps from compositions of $T_M$ with itself 
to $T_M.$ Each such composition preserves the image of $M,$ therefore there is  a unique map, from any  self-composition of $T_M,$ to the terminal object $T_M.$  Recall  all the conditions on a co-augmented to be a monad involve equality between various maps to the monad itself.  It follows that all the needed equalities are satisfied by $T_M.$  
\end{proof}
\medskip

We conclude that  the above basic properties  of $T_\DD$
holds when one replaces the inclusion $I:\DD\subseteq \CC$ with any  functor $M:\DD\to \CC.$ In this case, the  right Kan extension $T_M$ is the terminal monad on $\CC$ that  preserves the image subcategory of the given functor $M.$    In case  $\DD=\CC$  and  where the functor  $M: \CC\to \CC$ is  a co-augmented functor, we got  the terminal monad
 $T_M$ among those that preserve the image of $M.$

Consider the special, well-known case, where  $M$ is an idempotent localization functor $Id\to M\cong M^2. $ Namely, a projection onto a subcategory of $\CC.$ In that case, 
  $m: T_M\to M$ is an  equivalence. Namely, $M$ is its own terminal monad.  
(Compare: \cite{nlabidempotentmonad})
$$$$

For the sake of completeness, we state:
\begin{proposition}\label{localizationisterminal}
Let $L:\CC\to \CC$ be a co-augmented idempotent functor, i.e. localization- projection onto a full subcategory of local object. Then $L$ is its own terminal monad i.e. $L\cong T_L.$
\end{proposition}

\begin{proof}
First note that $L$ has associated  monad structure since $L\cong L^2$
by the two natural maps. Second, for every monad that preserves the image of $L,$ namely, with the natural map 
$L\to ML$ an equivalence (isomorphism) one gets a map 

$$M\to ML\to LML\cong L^2\cong L$$

The monad structure of $M$ forces the uniqueness of the map of monads $M\to L$ since any map of monads

$f: M\to L$  is a retract of 
$Mf: M^2\to ML\cong L,$ being a monad map, which is a retract of 
$M(id\to L.)$

$$$$

Starting with the map of monads \ref{1},
$$$$
\begin{equation} \label{1}
\begin{tikzcd}
& id  \ar[rd,"i_L"]\ar[ld,"i_M"'] 
 & \\
M \ar[rr,"f"']  & \  & L \end{tikzcd}
\end{equation}
$$$$
Applying  $M$ to this triangle of maps we see immediately that
 $f$ is uniquely determined by the monad

\end{proof}
$$$$

\begin{remark}
If the subcategory $\DD$ as above is closed under all limits then it is localizing 
and the $\DD$-completion 
$T_\DD$ is the localization 
$L_\DD: \CC \to \DD,$ projecting $\CC$ to the subcategory $\DD.$

\end{remark}
$$$$

\subsection{\texorpdfstring{$T_M$}{} in terms of \texorpdfstring{$M$}{}}

It turns out that there is a direct 
formula expressing the terminal monad $T_M,$ associated with $M,$ in terms of $M.$ 
Generally,
given a monad $Id\to M,$  
consider  a co-augmented functor $Id\to F$ that preserves the image of $M,$ i.e. $M\cong FM.$ Applying such an $F$ to $Id\to M$ we get a  canonical map $F\to M,$ for every such functor. In particular,
 for any monad $M,$
one gets  a natural map $T_M\to M,$
from the terminal monad to $M,$ giving rise to the augmentation $T\to Id$ of the functor $T.$ 

In the following  this last map is identified with the natural  map to $M,$ of the {\em equalizer of the natural diagram:}
$M\rightrightarrows M^2.$ In addition this map $T_M\to M$ is shown to be a map of  monads. 

\medskip

Let us start with two basic  properties:

\begin{proposition}\label{NisanMalg}
Let $f:(M, \mu_{{}_M})\to (N,\mu_{{}_N})$ be a map of monads.
Then the monad $N$ is naturally an
$M-$algebra. In particular, $T_M(N)\cong N,$ thus $f$ induces a map of monads $Tf: T_M\to T_N.$ Hence $T:Mon\CC\to Mon\CC,$ has a natural structure of {\em augmented} endo-functor $T\to Id$ on the category $Mon\CC$ of monads over $\CC.$

\end{proposition}

\begin{proof}
Since $M$ is an algebra over itself, we need to show that
the natural map  $N\to M(N),$   gotten by applying the 
co-augmentation $\iota_M: Id\to M$ to $N,$ has a left inverse i.e., that $N$ is a retract of $M(N)\equiv M\circ N.$ 
Since $T_M$ preserves $M,$ it preserves also any retract of $M.$ 
The natural left inverse  is given by the composition:

$$ MN \xrightarrow{f\circ N} NN\xrightarrow{\mu_{{}_N}} N $$

Hence, $N(x)$ has been shown to be a retract of $M(N(x)),$ for all 
objects $x\in obj \CC.$
Thus $T_M(N(x))\cong N(x).$ But $T_N$ is the terminal monad that preserves all objects of the form $N(x).$ Therefore, there is a unique map $T_M\to T_N,$ as needed.
\end{proof}

Second, an interesting closure property 

%that follows directly from a %theorem of Libman  in  %\cite{libmanuniversal} Theorem %4.7:

\begin{proposition} \label{preservelimit}
 Let $X: I\to \CC$ be a functor with $I$  being a small (indexing) category and $M$ a monad over $\CC.$ 
 Assume that for each $i\in I$ the object $X_i\in obj \CC,$ is a retract of $M(X_i.)$
 Then the object  $Y=lim_IX_i $ is a retract of $T_M(Y),$ hence $T_{{T_M}}(Y)\cong Y.$ Similarly, any such limit  $Y$ of $M-$algebras is naturally a $T_M$- algebra. 
\end{proposition}

Note: there is no assumption here about  relations among the various retractions, namely the retract structures on different $X_i.$ Thus the limit is not in general an $M$-retract but it is a $T_M$
retract. For example, (in  $Cat_\infty$) in the infinity category of spaces, if $M=R$ is the free   $R$-module spanned by a space $X$, then a limit of any diagram of such $R$-GEMs is not, in general, a $R$-algebra but rather a $R_\infty =T_R$-algebra.

\begin{proof}   

Consider the composition:

$$
lim_i X_i\to T_M(lim_i X_i)\overset a {\to} lim_i T_M(X_i)\cong lim_i X_i
$$
which is clearly the identity map.

 The right-hand side map $a$ is the assembly map for limits, and the equality on the right is a consequence of $T_M(X_i)\cong X_i, $ since the latter is a retract 
 of $M(X_i)$ by assumption and hence also preserved by $T_M.$

The right-hand side map $a,$ is directly seen to 
equip $lim_i X_i$ with an 
$T_M-$ algebra structure.
  
\end{proof}
$$$$

\subsection{The equalizer as the terminal monad.}\label{equalizer}
\medskip
Given a co-augmented functor $Id\to M,$ denote by  $Eq_M$ the equalizer of the two natural maps  $M\rightrightarrows M^2$ coming form the  co-augmentation structure $Id\to M.$ 

First, we note the following:
\begin{lemma}\label{maptoequalizer}
Let $Id\to M$ be an co-augmented functor $\CC\to \CC.$ Let $F$  be any co-augmented functor $F:\CC\to \CC$ that preserves $M,$ i.e. with
$(Id\to F)(M)=M\overset \cong \to F(M)$ an equivalence.
Then there is a natural map, of co-augmented functors, $F\to Eq_M$
from $F$ to the equalizer of $M\rightrightarrows M^2.$ 

\end{lemma}
\begin{proof}
 Apply $F$ to the commutative diagram
 $Id\to M\rightrightarrows M^2,$ to get the desired factorization to the equalizer:
 observing that we get commutative:
 $$
 F\to F(M) \rightrightarrows F(M^2)
 $$

 which  is equivalent to:
$$
 F\to M \rightrightarrows M^2
 $$
 by the assumption on $F.$  Hence there is a well-defined factorization of the left-hand side map
 through the equalizer $F\to Eq_M.$ This map clearly respects the co-augmentation $Id\to F\to M.$
\end{proof}

$$$$
For the rest of the discussion, we will mostly assume, sometimes for convenience only, that $M$ is a monad on $\CC.$ We noticed that
every functor that preserves 
$M$ maps naturally to the functor $Eq_M.$
The same is true in particular to the terminal functor that preserves $M.$ But $Eq_M$  preserves $M$ thus, by definition, it maps uniquely to the terminal $T_M.$
This brings us to the following:
$$$$
\begin{proposition}\label{equalizerequaterminal}
Let $M$ be a monad in $Mon\CC,$ the category of monads over  $\CC.$
The terminal monad preserving the image of $M$  is
naturally equivalent to the equalizer:

$$
T_M\cong Eq(M \rightrightarrows M^2.)
$$
Equalizer that is, with respect to the two natural transformations given by the augmentation. In particular, the equalizer itself has a natural structure of a monad. 
\end{proposition}

\begin{proof}

Observe  that the diagram of  maps $M \rightrightarrows M^2$ is not, in general,
a diagram of monads. The maps in it are natural transformations of co-augmented functors. Thus the equalizer is a co-augmented functor, but it is not immediately clear why it is a monad. Fakir states this without proof. It does follow below 
from the observation  that the equalizer is naturally isomorphic to the terminal monad 
$T_M.$

 To prove that, note  that $Eq_M$ preserves $M,$ 
 i.e. $Eq_M(M)\cong M,$  were here
 $M$ denotes the underlying co-augmented functor of the monad $M.$ 
 The reason is that clearly  as co-augmented functors, there is an equivalence $M\cong tot M^\bullet (M)=lim_{\Delta} (M^\bullet (M)),$ since the latter co-simplicial functor  has an extra co-degeneracy map. But in 1-category,  $tot M^\bullet\cong Eq_M.$
 Therefore there is a unique map of co-augmented functors $Eq_M\to T_M.$

 First, we prove that this
 map is an equivalence of functors. This will endow the equalizer with a monad structure.
$$$$

The monad structure on $T_M$ comes from the natural map $T_MT_M\to T_M$ given by 
the universal property of the range. Since the identity is only  self-map $T_M \to T_M,$ the last map satisfies the monad conditions.

%More explicitly by  the assembly map, using the above result $Eq_M\cong T_M$:

%$$
%$$ T_MT_M=Eq_MEq_M\to Eq(Eq_M(M) %\rightrightarrows  Eq_M(M^2))=Eq(T_M(M) \%rightrightarrows  T_M(M^2))=T_M.
%$$

$$$$
Now consider the following  maps (=natural transformations) of co-augmented  functors:

$$
Eq_M \overset q\longrightarrow T_M\overset t \longrightarrow Eq_M\overset q\longrightarrow T_M
$$
The map $q$ is uniquely guaranteed by the equation $Eq_M\circ M\cong M,$ since $T_M$ is terminal {\em co-augmented functor} with this property.

The map $t$ is given by the universal property  of $Eq_M.$ See lemma \ref{maptoequalizer} above. Namely, since the monoid $T_M$ preserves $M,$ ($T_M\circ M\cong M$) so it preserves  the  co-simplicial object $M^\bullet.$ 
 When we apply $T_M$ to $Id\to M^\bullet$  we get a map of $T_M$ to the limit of $M^\bullet$, which in our case is the equalizer $Eq_M.$ 

In the  above composition of three natural transformations, the induced self-map of $T_M$ is the identity since $T_M$ is a terminal object.
We claim that the composition 
$t\circ q$ is equivalent to the identity.

To see that, consider the diagram:

$$$$
\[
\xymatrix{
Eq_M\ar[d]\ar[rr]^q&&T_M\ar[d]\ar[rr]^t&& Eq_M\ar[d]\\
Eq_M\circ M\cong M\ar@<-0.5ex>[d]\ar@<0.5ex>[d]^g\ar[rr]^\cong &&T_M\circ M\ar@<-0.5ex>[d]\ar@<0.5ex>[d]^f\ar[rr]^\cong &&M \ar@<-0.5ex>[d]\ar@<0.5ex>[d]^g\\
Eq_M\circ M^2\cong T_M\circ M^2\ar[rr]^\cong &&T_M\circ M^2\ar[rr]^\cong &&M^2}
\]
$$$$

gotten by applying the above maps: $Eq_M\to T_M\to Eq_M$ to the 

natural transformations:

$$
Id\to M\rightrightarrows M^2.
$$

Since all the horizontal  maps in diagram that involve $M,M^2,$ and denoted by $"\cong" $ are equivalences, the required self-map
on the equalizer, $Eq_M$ is also an equivalence as needed. Thus we conclude that the maps $t,s$ are 
equivalences of co-augmented functors. It follows that $Eq_M$
has a structure of a monad coming from that of $T_M,$ and the
two are equivalent as a monad, as stated.
\end{proof}

$$$$

The terminal monad  $T_M,$ has the following properties:
$$$$

\begin{theorem}\label{monadmapTMtoM}
For any  monad $M,$ the above map $T_M\to M$, \ref{monadmapTMtoM},  is a map of monads. 
The assignment $M\mapsto T_M$
gives an  endo-functor $MonC\to Mon C,$
together with a natural transformation $T\to Id,$ namely, an {\em augmented endo-functor} $T$ on $Mon C.$

\end{theorem}

\begin{proof}

 By   \ref{NisanMalg} above,  $T_M $  is natural in the variable $M.$ Or, since $Eq_M$  is functorial in $M,$  so $T_M$ is by the above theorem \ref{equalizerequaterminal}. Hence the assigment:  $T: M\mapsto T_M$ defines an endo-functor on $Mon\CC:$
For a given  map of monads $M\to M'$
one  gets a map of co-simplicial
resolutions $M^\bullet\to M'^\bullet.$
Or, since $T_M$ is naturally equivalent to the  equalizer 
of $M\rightrightarrows M^2,$ we get a well-defined map on the terminal monads 
 with a natural map $m_M: T_M\to M,$ as needed. 

We now prove that this  map $T_M\to M$
is a map  of monads, namely, respect the monad structure
$M^2\to M.$

Now since $Eq_{(-)}$ is a co-augmented functor we get
a commutative diagram involving $Eq_{M^2}$ by applying this  functor to $M\rightrightarrows M^2.$
The natural   map $Eq_MEq_M\to Eq_{M^2,}$
completes the argument, giving the necessary commutation of the monad  structures. 

In more detail, for any map of co-augmented functors, such as $m_M: T_M\to M,$
the corresponding co-faces maps  $M\rightrightarrows M^2$ etc. commutes
with $m_M$ and $m_{T_M.}$

Consider the natural diagram:

$$$$
\[
\xymatrix{
T_M\ar[d]^=&&T_MT_M\ar[ll]_{\mu_1}
\ar[d]^{\mu_3}\\
{T_M\cong Eq_M}\ar[d]^{m_{M}=\iota_M} &&Eq_{M^2}\ar[ll]_{\mu_2}\ar[d]^{\iota_{M^2}} %M\ar@<-0.5ex>[d]\ar@<0.5ex>[d]^f\ar[rr]^\cong &&M \ar@<-0.5ex>[d]\ar@<0.5ex>[d]^g
\\
M && M^2\ar[ll]_{\mu}
}
\]
$$$$

To prove that $T_M\to M$ is a map of monads, one only needs to show that the monad 
 structures of $M$ and $T_M,$ written as $\mu,\mu_1,$   are compatible with the natural map $T_M\to M.$ Namely, that the two  composition arrows, involving $\mu$ and $\mu_1,$
 of $T_MT_M\to M$
 in the diagram below are equal. Namely, that the outer square of the maps below commutes.
 
 Consider the diagram below:
The arrows $\mu_1,\mu_3$ are defined by the universal properties of $T_M$ and the equalizer, correspondingly,   since their common 
domain preserves $M$ and $M^2,$  using \ref{lemma:equalizer} above.  The map ${\iota_{M^2}}$ is the inclusion of the equalizer. 
%In addition by uniqueness, one has: $\mu_2=T(\mu).$ 
%The map $\mu_3$
%is defined again by the universal property of its range, being an equalizer since its domain preserves $M$ and $M^2.$

Note that the bottom square commutes since here the equalizer  $Eq_{(-)}$ is considered  here as a co-augmented functor,$Eq_{(-)}\to Id,$  from the category of co-augmented functors over $\CC$ to itself. 

The top square below commutes by the terminal property of its bottom left corner $T_M,$ admitting only one map from the functor $T_MT_M$ at the top right corner.

Thus the whole diagram commute as needed.

\end{proof}

$$$$
\subsubsection{Example} Consider the ultrafilter monad $X\to U(X)$ discussed elsewhere in this paper. The associated terminal monad $T_U$ preserves the image of $U$ which includes sets of arbitrary cardinality. Hence it preserves all sets and must be the identity monad. In fact, the equalizer $Eq_U$ is easily seen to be the identity functor on sets, as it includes only principal ultrafilters. On the other hand, the terminal monad associated to the 
profinite completion  $G\to pro-G\equiv \widehat G$ in the category of groups is not the identity functor since 
on a finitely presented group $T_{pro}(G)\cong  pro-G$ since the completion is idempotent on these groups so that the equalizer of
$pro-G\rightrightarrows pro-(pro-G) $ is $pro-G$ itself, the completion of $G.$

$$$$

\section{Explicit expressions for \texorpdfstring{$T_d$}{}}
Here we explicitly express the terminal monad, associated with an object $d\in \CC,$ as a
structured double dual, see below, as opposed to the usual one as in \cite{adamek2021d}, section 2.

$$$$

For a set $S$ and an object $c\in {\sf Ob}(\CC)$ we denote by $c^S$ the product of copies of $c$ indexed by $S$
\begin{equation}
   c^S = \prod_{s\in S} c.
\end{equation}
The projections will be denoted by ${\sf pr}_s: c^S\to c.$
A morphism $f:c'\to c^S$ is defined by a family of morphisms $f_s={\sf pr}_s\circ  f:c'\to c$ which are called components of $f.$ The object $c^S$ is contravariant by $S.$
More precisely, this defines a functor 
\begin{equation}
\CC \times {\sf Set}^{\sf op}\to \CC,\hspace{1cm} (c,S)\mapsto c^S    
\end{equation}
such that, if $f:S\to S'$ is a function, then the morphism $c^f:c^{S'}\to c^{S}$ is defined so that $(c^f)_s={\sf pr}_{f(s)}.$

\begin{proposition}
Let $\CC$ be a  category  with limits and $\DD$ be a small full subcategory of $\CC.$ Then the functor of $\DD$-completion exists and it is given by the end
\begin{equation}
T(c) = \int_{d\in \DD} d^{\CC(c,d)}.  
\end{equation}
It can be also presented as an equalizer 
\begin{equation}
T(c) = {\sf eq}\left( \prod_{d\in \DD} d^{\CC(c,d)}  \rightrightarrows \prod_{\alpha:d_1\to d_2} d_2^{\CC(c,d_1)} \right),
\end{equation}
where the first morphism is induced by $ d_1^{\CC(c,d_1)} \xrightarrow{\alpha^{\CC(c,d_1)}} d_2^{\CC(c,d_1)} $ and the second morphism is induced by $d_2^{\CC(c,d_2)} \xrightarrow{d_2^{\CC(c,\alpha)}} d_2^{\CC(c,d_1)}.$
\end{proposition}
\begin{proof}
Follows from the interpretation of right Kan extensions in terms of ends \cite[Ch. X, \S 4, Th.1]{mac2013categories} and the characterisation of ends in terms of equalisers \cite[Remark 1.2.4]{loregian2015coend}. 
\end{proof}

\begin{corollary}\label{cor:one_object}
Let $\CC$ be a complete category and $\DD=\{d\}$ the full subcategory consisting of one object. Then $T_d=T_{\DD}$ exists and
\begin{equation}
T_d(c) \cong {\sf eq}\left(  d^{\CC(c,d)}  \rightrightarrows  \left(d^{\CC(c,d)}\right)^{ {\sf End}_\CC(d)} \right).
\end{equation}
\end{corollary}
$$$$
\subsubsection{Remark.}   
There is an alternative description of $T_d$ using the double dual monad. Denote by  $DD_d(c)$ the "naive double dual"  monad $c\to d^{\CC(c,d)}=\Pi_{c\to d}d.$ Then
it is not hard to see using  the arguments in \ref{equalizer} that  $T_d$ is equivalent to the terminal monad $T_{DD_d}$
associated with $DD_d.$  Notice that $DD_d(d)$ is isomorphic to $d^l$ for some $l\geq 1.$ Hence, since  the terminal  $T_{DD_d}$ preserves the image of ${DD_d},$  $T_{DD_d}$ also preserves its retract
$d,$ and one has a unique map of monads $Eq_{DD_d}\cong T_{DD_d}\to T_d.$ Similarly there is a map in the other direction: Namely, the desired  element in the monad category $Mon\CC$:

$$Mon(T_d,DD_d)=
Mon(T_d, d^{\CC(c,d)})$$

is determined by the composition of maps in a $\CC$:

$$T_d(c)\times \CC(c,d)\to T_d(c)\times \CC (T_d(c), T_d(d)=d) \overset {eval}\longrightarrow  d.$$\\

Or, stated otherwise, the map $T_d(c)\to \Pi_{c\to d}d$ is give factor-wise by: $T_d(c)\to d=  T_d(c\to d),$  since $T_d(d)\cong d.$ This also  can serve  to prove  the crucial property of $T_d$ namely, for any map $x\to y$
in $\CC$  with $\CC(y,d)\cong \CC(x,d)$
satisfies $T_d(x)\cong T_d(y);$ which is of course evident from the explicit expression for $T_d$ above. It is clear, though not expanded here, that, in the case where the category $\CC$ is enriched over itself, the above approach works well, using internal  hom objects $hom(hom(-,d),d).$

\section{Examples in the category of sets and groups}

\subsection{Variations on the set of ultrafilters.}
For any set $X$ we denote by $\PP(X)$ the set of all subsets of $X$. We treat $\PP$ as a (contravariant) functor
\begin{equation}
    \PP:{\sf Set}^{\sf op} \longrightarrow {\sf Set},
\end{equation}
where, for a map  $f:X\to X',$ the map $\PP(f):\PP(X')\to \PP(X)$ is defined as $\PP(f)(Y)=f^{-1}(Y).$ Note that the characteristic function defines an isomorphism 
\begin{equation}
\chi: \PP(X) \cong {\sf Set}(X,2).    
\end{equation}
The composition 
\begin{equation}
\PP \PP : {\sf Set} \longrightarrow {\sf Set}    
\end{equation}
is a (covariant) functor that has a natural coaugmentation 
\begin{equation}
    \eta_X : X\longrightarrow \PP(\PP(X))
\end{equation}
such that $\eta_X(x)$ is the set of all sets $Y\subseteq X$ containing $x.$ If we denote by $\UF(X)$ the set of ultrafilters on $X,$ we obtain that $\UF$ is an co-augmented sub-functor of $\PP\PP$
\begin{equation}
    \UF\subseteq \PP\PP.
\end{equation}

Recall that $\UF(X)$ can be thought of as the underlying set of the Stone-\v{C}ech compactification of the set $X.$

$$$$

Let us define another co-augmented sub-functor of $\PP\PP.$ An element $A\in \PP(\PP(X))$ is called \emph{ultraset}, if 
\begin{itemize}
    \item[(US1)] $\emptyset \notin A;$
    \item[(US2)] for any $Y\subseteq X$ one and only one of the sets $Y,X\setminus Y$ is an element of $A.$
\end{itemize}
Note that the axiom (US1) can be equivalently replaced by 
\begin{itemize}
    \item[(US1')] $X\in A.$
\end{itemize}

\begin{example}
Any ultrafilter is an ultraset.
\end{example}
\begin{example} Let  $X$ is a finite set of an odd cardinality $|X|=2n+1$ and
\begin{equation}
A= \{Y\subseteq X\mid |Y|\geq n+1\}.    
\end{equation}
Then $A$ is an ultraset on $X$ which is not an ultrafilter  for $n\geq 1.$ 
\end{example}

\begin{lemma}\label{lemma:ultrafilter}
Let $A$ be an ultraset on a set $X.$ Then $A$ is an ultrafilter if and only if  for any partition into three disjoint subsets $X=P_0\sqcup P_1 \sqcup P_2$, there exists a unique $i\in \{0,1,2\}$ such that $P_i\in A.$ 
\end{lemma}
\begin{proof} Assume that  $A$ is an ultrafilter and $X=P_0\sqcup P_1 \sqcup P_2$ is a partition. If there exists $i$ such that $P_i\in A,$ then it is obviously unique because $A$ is closed under finite intersections.  Let us prove that it exists. Assume the contrary that $P_i\notin A$ for any $i.$ Then $X\setminus P_i \in A,$ and hence $P_0=(X\setminus P_1) \cap (X\setminus P_2)\in A,$ which is a contradiction.

Now assume that for any partition into three disjoint subsets $X=P_0\sqcup P_1 \sqcup P_2$, there exists a unique $i\in \{0,1,2\}$ such that $P_i\in A.$ In order to prove that $A$ is an ultrafilter, we need to prove that: (1)  $Y\in A$ and $Y\subseteq Y'\subseteq X$ implies $Y'\in A;$ (2) $Y, Y'\in A$ implies $Y\cap Y'\in A.$ 

Let us prove (1). Take $P_0=Y,$ $P_1=Y'\setminus Y$ and $P_2=X\setminus Y'.$ Then $P_0\in A,$ and hence, $P_2\notin A.$ By (US2) we obtain  $X\setminus P_2=Y'\in A.$ 

Let us prove (2). In the proof, we use that we already proved (1).  Take $P_0=Y\cap Y',$ $P_1=Y\setminus Y',$ $P_2=X\setminus Y.$ Since $Y\in A,$ we have $P_2\notin A.$ Therefore either $P_0\in A,$ or $P_1\in A.$ We need to prove that $P_0\in A.$ Assume the contrary that $P_1\in A.$ Since $P_1 \subseteq  X\setminus Y',$ using (1), we obtain $X\setminus Y'\in A.$ It follows that $Y'\notin A,$ which is a contradiction. Hence $P_0\in A.$ 
\end{proof}

The set of all ultrasets on $X$ is denoted by $\US(X).$   It is easy to check that $\US$ is a co-augmented sub-functor of $\PP\PP$
\begin{equation}
\UF \subseteq \US \subseteq \PP\PP.
\end{equation}

Let $n$ be a natural number, taken  as  an ordinals $n=\{0,\dots,n-1\}.$ We denote by $T_n:{\sf Set}\to {\sf Set}$ the functor of $\{n\}$-completion i.e. it is the terminal co-augmented functor with the property that $n\to T_n(n)$ is an isomorphism  $n.$ 

\begin{lemma}
Let ${\sf Fin}_{\leq n}$ denotes the class of finite sets of cardinality at most $n.$ Then 
\begin{equation}
    T_n = T_{{\sf Fin}_{\leq n}}.
\end{equation}
\end{lemma}
\begin{proof}
Since any set of cardinality at most $n$ is a retract of $n,$ this follows from Lemma \ref{lemma:retract}.
\end{proof}

\begin{proposition}\label{prop:US}
The co-augmented functor of $2$-completion on the category of sets is isomorphic to  $\US$
\begin{equation}
    T_2 \cong \US. 
\end{equation}
\end{proposition}
\begin{proof}
By Corollary \ref{cor:one_object} we see 
\begin{equation}
    T_2(X)= 
    {\sf eq}\Big( {\sf Set}( {\sf Set}(X,2),2) 
    \rightrightarrows 
    {\sf Set}( {\sf Set}(X,2),2)^{{\sf End}(2)}\Big).
\end{equation}
The characteristic function defines a bijection
$\PP(X)\cong {\sf Set}(X,2).$
There are four maps $2\to 2:$ (1) the identity map ${\sf id}=e_1;$ (2) the map $e_2$ sending all to $0;$ (3) the map $e_3$ sending all to $1;$ (4) the permutation $e_4.$ The composition with them correspond to four maps on $f_i^X: \PP(X)\to \PP(X):$ (1) $f_1^X(Y)= Y$; (2) $f_2^X(Y)=\emptyset;$ (3) $f_3^X(Y)= X;$ (4) $f_4^X(Y)=X\setminus Y.$
Consider the isomorphism 
\begin{equation}
\PP(\PP(X))\cong {\sf Set}({\sf Set}(X,2),2).    
\end{equation}
So we need to prove that 
\begin{equation}
    \US(X) = {\sf eq}\Big( \PP(\PP(X)) \rightrightarrows \PP(\PP(X))^{{\sf End}(2)} \Big)
\end{equation}
The equaliser consists of such elements $A\in\PP(\PP(X))$  that the equation $f_i^{\PP(X)}(A)=\PP(f^X_i)(A)$ is satisfied for any $i.$ For $i=1$ it is satisfied for any $A$. For $i=2$ we have $f_2^{\PP(X)}(A)=\emptyset$ and
\begin{equation}
\PP(f^X_2)(A)=(f_2^X)^{-1}(A)= \begin{cases}
\PP(X), & \emptyset \in A;\\
\emptyset,& \emptyset \notin A.
\end{cases}
\end{equation}
Then it is satisfied for $i=2$ iff $ \emptyset\notin A$ (axiom (US1)). Similarly, we obtain that the equation is satisfied for $i=3$ iff $X\in A$ (axiom (US1')). For $i=4$ we have that $f_4^{\PP(X)}(A)=\PP(X)\setminus A$ and $\PP(f_4^X)(A)=\{X\setminus Y\mid Y\in A \}.$ Then the equation is satisfied for $i=4$ iff the axiom (US2) is satisfied. 
\end{proof}

\begin{proposition}
Let ${\sf Fin}$ denote the full subcategory of ${\sf Set}$ consisting of finite sets. Then $T_{\sf Fin}$ is isomorphic $T_3$ and isomorphic  to  $\UF$
\begin{equation}
T_{\sf Fin}\cong T_3\cong \UF.
\end{equation}
\end{proposition}
\begin{proof} It is well-known that $\eta: K\to \UF(K)$ is an isomorphism for any finite $K.$ So it is enough to prove that $\UF$ is the terminal among all co-augmented functors $(F,\eta^F)$ such that $\eta^F: K\to F(K)$ is an isomorphism for any set $K$ such that $|K|\leq 3.$ By the universal property of $\US=T_2$ (Proposition \ref{prop:US}) we see that there is a unique morphism of co-augmented functors $\varphi: F\to \US.$ So we just need to prove that for any set $X$ the image of $\varphi_X$ is in $\UF(X).$ Denote by $F'$ the image of $\varphi.$ Note that $F'$ is a co-augmented sub-functor of $\US.$ So we need to prove that $F'(X)\subseteq \UF(X).$ 

Since $\eta:K\to F(K)$ is an isomorphism for finite any $K$ such that $|K|\leq 3$ we see that 
\begin{equation}
\UF(K)=F'(K).
\end{equation}
for any $K$ such that $|K|\leq 3.$ 

Let us prove that $F'(X)\subseteq \UF(X).$ Take an ultraset $A\in F'(X).$ Consider a partition $X=P_1\sqcup P_2 \sqcup P_3.$ Define a map $\alpha:X\to 3$ such that $\alpha^{-1}(i)=P_i.$ The map $F'(\alpha):F'(X)\to F'(3)=\UF(3)$ sends $A$ to $\eta(i_0)\in \UF(3)$ for some $i_0\in \{0,1,2\}.$ Since  $\{i_0\}\in \eta(i_0)$ and for $i\ne i_0$ we have $\{i\}\notin \eta(i_0),$ we obtain that $P_{i_0}=\alpha^{-1}(i_0)\in A$ and $P_i=\alpha^{-1}(i)\notin A$ for $i\ne i_0.$ Therefore the assumption of Lemma  \ref{lemma:ultrafilter} is satisfied, and hence, $A$ is an ultrafilter.   
\end{proof}

\subsection{Examples: Groups and  modules}\label{groupsmodules}
Here we briefly consider the examples alluded to in the first section.  The examples below 
are proved by applying the expression \ref{limitovercat} above. It is rather immediate to see that by taking $\DD\subset \CC$ to be the subcategory of finite groups in the category of all groups, the $\DD-$completion
 $T_\DD$ is canonically  isomorphic to the (discrete!) pro-finite completion  functor on groups.

 Similarly, when $\DD\subset \CC$ is the subcategory of nilpotent groups in the category of all groups. Similarly, for the 
 completion of an $A-$module $M,$ with respect to 
 an ideal $I\subseteq A$ in a ring $A:$ Namely,
 $M\to lim M/I^k M.$

 In  the above example, the fact that 
 $\DD$ is  a large category can be dealt with by noticing that  for each  $A-$module  $M,$ the tower of quotients $ M\to (M/I^kM)_k$ is co-final in the category  $M\downarrow \DD,$
 appearing in \ref{limitovercat}.
 
 In case the ring $A=K$ is a field 
 the usual double dual  functor of a $K-$ vector
 space $V\to V^{**}$ appears as a terminal 
 monad $T_K,$ since the double dual  in \ref{eq:enddual} above is reduce here to 
 $V^{**}.$

\section{Completions and operads} \label{operadiccompletion}

In this section, we continue to assume that $\CC$ is closed under  limits. 

\subsection{Objects with an action of a monoid} Let $M$ be a monoid. An $M$-object in $\CC$ is an object $c$ endowed by a homomorphism of monoids $f^c: M\to {\sf End}(c).$ If $X$ is an $M$-set and $c$ is an $M$-object, we define the \emph{hom-object} over $M$ as an equalizer 
\begin{equation}
    {\sf hom}_M(X,c) = {\sf eq}(\sigma,\tau: c^X \rightrightarrows (c^X)^M),
\end{equation}
where $\sigma_m=c^{f^X(m)}$ and $\tau_m=(f^c(m))^X$ for any $m\in M.$ If $\CC$ is a category of sets, $\hom_M(X,c)$ coincides with the ordinary hom-set in the category of $M$-sets.  

For any  two objects $c,d$ from $\CC$ the hom-set $\CC(c,d)$ has a natural structure of ${\sf End}(d)$-set defined by the composition. Then Corollary \ref{cor:one_object} can be reformulated as
\begin{equation}\label{eq:enddual}
    T_d \cong \hom_{{\sf End}(d)}(\CC(-,d),d).
\end{equation}

\subsection{Objects with an action of an operad}\label{operad}

Let $O$ be an operad (of sets). For an object $c$ we denote by ${\sf O}(c)$ the endomorphism operad of $c,$ whose $n$-th component is  ${\sf O}(c)_n=\CC(c^n,c).$ An $O$-algebra in $\CC$ is an object $c$ endowed by a morphism $f^c:O\to {\sf O}(c).$ If $X$ is an $O$-algebra in the category of sets and $c$ is an $O$-algebra in $\CC,$ we defined the hom-object over $O$ as an equaliser 
\begin{equation}
  \hom_O(X,c)={\sf eq}(\sigma,\tau: c^X \rightrightarrows \prod_{n=0}^\infty (c^{X^n})^{O_n}),
\end{equation}
where $\sigma$ and $\tau$ are defined so that $\sigma_{n,o}=c^{f^X_n(o)}:c^X \to c^{X^n}$ and
\begin{equation}
\tau_{n,o,x_1,\dots,x_n}= f^c_n(o) \circ ({\sf pr}_{x_1},\dots,{\sf pr}_{x_n}): c^X \to c    
\end{equation}
for any $n\geq 0,$ $o\in O_n$ and $x_1,\dots,x_n\in X.$  Here we denote by $({\sf pr}_{x_1},\dots,{\sf pr}_{x_n}):c^X \to c^n$ the morphism with components ${\sf pr}_{x_i}.$ For the special case 
$n=0$ we have $X^n=1=\{0\},$ $c^n=1,$ and $\tau_{0,o}:c^X\to c$ is the composition of $c^X\to 1$ and $f_0^X(o):1\to c.$
Note that 
\begin{equation}
 \sigma_{n,o,x_1,\dots,x_n} = {\sf pr}_{f^X_n(o)(x_1,\dots,x_n)}: c^X \to c.   
\end{equation}

For any $n\geq 0$ we also consider 
\begin{equation}
 \hom^{n}_O(X,c) ={\sf eq}(\sigma_{n},\tau_{n}: c^X \rightrightarrows  (c^{X^n})^{O_n}) 
\end{equation}
and 
\begin{equation}
 \hom^{\leq n}_O(X,c) ={\sf eq}(\sigma_{\leq n},\tau_{\leq n}: c^X \rightrightarrows  \prod_{i=0}^n (c^{X^i})^{O_i}). 
\end{equation}
The projection $\prod_{i=0}^n (c^{X^i})^{O_i}\to \prod_{i=0}^{n-1} (c^{X^i})^{O_i}$ induces a morphism 
\begin{equation}
    \hom^{\leq n}_O(X,c) \longrightarrow \hom^{\leq n-1}_O(X,c).
\end{equation}

Since $\prod_{i=0}^{\infty} (c^{X^i})^{O_i} =\underset{n}\varprojlim\  \prod_{i=0}^n (c^{X^i})^{O_i},$ using that limits commute with limits, we obtain 
\begin{equation}
\hom_O(X,c) = \underset{n}\varprojlim\  \hom_O^{\leq n}(X,c).    
\end{equation}

\subsection{Completion with respect to a power}

For an object $d$ of $\CC$ we denote by ${\sf O}^1(d)$ the suboperad of the endomorphism operad ${\sf O}(d)$ such that 
\begin{equation}
    {\sf O}^+(d)_0=\emptyset, \hspace{1cm} {\sf O}^+(d)_n={\sf O}(d)_n
\end{equation}
for $n\geq 1.$

For any two objects $c,d$ of $\CC$ there is a natural structure  of ${\sf O}^1(d)$-algebra on the set $\CC(c,d):$ 
for any $\alpha:d^n\to d$ we consider
the map 
\begin{equation}
f(\alpha) : \CC(c,d)^n \cong \CC(c,d^n) \xrightarrow{\CC(c,\alpha)} \CC(c,d). 
\end{equation}

For any object $d$ of $\CC$ we consider the functor $T_{d^n}$ of $d^n$-completion. We also consider 
\begin{equation}
T_{d^+}=T_{\{ d^n\mid n\geq 1 \}}, \hspace{1cm}  T_{d^\bullet}=T_{\{ d^n\mid n\geq 0\}}. 
\end{equation}

This subsection is devoted to the proof of the following theorem.

\begin{theorem}\label{th:powers}
Let $\CC$ be a complete category and $d$ be its object. Then for $n\geq 1$, there are isomorphisms of co-augmented functors
\begin{equation}
T_{d^n}\cong \hom_{{\sf O}^+(d)}^n(\CC(-,d),d)\cong \hom_{{\sf O}^+(d)}^{\leq n}(\CC(-,d),d),    
\end{equation}
\begin{equation}
T_{\{1,d^n\}}\cong \hom_{{\sf O}(d)}^{\leq n}(\CC(-,d),d), 
\end{equation}
\begin{equation}
T_{d^+}\cong \hom_{{\sf O}^+(d)}(\CC(-,d),d),
\end{equation}
\begin{equation}
T_{d^\bullet}\cong \hom_{{\sf O}(d)}(\CC(-,d),d),
\end{equation}
where the augmentations of the right-hand functors are induced by  morphism to the element $d$ raised to the power ${\CC(c,d)}$, given by: $\tilde \eta_c: c\to d^{\CC(c,d)}$ with components $(\tilde \eta_c)_\alpha = \alpha.$
\end{theorem}

\begin{remark} \label{remarkmandel}
    
 An analog and potentially a special case of this formula, within the $\infty$-category of simplicial sets, appears in Mandell's theorem, \cite{Mandellpadichomotopy} and \cite{Mandellthmexposition}, Proposition 4.4. Here, the operadic double-dual appears as a version of homological p-completion. It gives the terminal functor that preserves certain $p$-adic Eilenberg-MacLane spaces.
\end{remark}
$$$$

In order to prove this theorem, we need to prove several lemmas.

\begin{lemma}\label{lemma:phi_equal} For $n\geq 0$ and a morphism $\varphi:e \to d^{\CC(c,d)},$
the diagram 
\begin{equation}
    e \xrightarrow{\varphi} d^{\CC(c,d)} \underset{\tau_n}{\overset{\sigma_n}{\rightrightarrows}} \Big(d^{\CC(c,d)^n}\Big)^{\CC(d^n,d)}
\end{equation}
is commutative ($\tau_n\varphi=\sigma_n\varphi$) if and only if for any morphism $\alpha:d^n\to d$ and any morphism $\beta:c\to d^n$ we have
\begin{equation}
    \varphi_{\alpha \circ \beta} = \alpha \circ (\varphi_{\beta_1},\dots , \varphi_{\beta_n}).
\end{equation}
\end{lemma}
\begin{proof}
The components of the morphisms $\tau_{n,\alpha},\sigma_{n,\alpha}$ are  
$\sigma_{n,\alpha,\beta_1,\dots,\beta_n}  = {\sf pr}_{\alpha \circ \beta } 
$
and 
$\tau_{n,\alpha,\beta_1,\dots,\beta_n}  = \alpha\circ ({\sf pr}_{\beta_1}, \dots, {\sf pr}_{\beta_n}).
$ The assertion follows. 
\end{proof}

\begin{lemma}\label{lemma:eta_equal} For $n\geq 0$ the diagram 
\begin{equation}
    c \xrightarrow{\tilde \eta_c} d^{\CC(c,d)} \underset{\tau_n}{\overset{\sigma_n}{\rightrightarrows}} \Big(d^{\CC(c,d)^n}\Big)^{\CC(d^n,d)}
\end{equation}
is commutative ($\sigma_n\tilde \eta_c=\tau_n\tilde \eta_c$), where $(\tilde \eta_c)_\alpha=\alpha.$
\end{lemma}
\begin{proof}
It follows from Lemma \ref{lemma:phi_equal}. 
\end{proof}

\begin{lemma}\label{lemma:equalizer}
For $n\geq 0$ the diagram
\begin{equation}
    d^n \xrightarrow{\tilde \eta_{d^n}} d^{\CC(d^n,d)} \underset{\tau_n}{\overset{\sigma_n}{\rightrightarrows}} \Big(d^{\CC(d^n,d)^n}\Big)^{\CC(d^n,d)}
\end{equation}
is an equalizer.
\end{lemma}
\begin{proof} By Lemma \ref{lemma:eta_equal} we have $\sigma_n\tilde \eta = \tau_n\tilde \eta.$
Let  $\varphi: e\to d^{\CC(d^n,d)}$  be a map that equalizes $\sigma_n$ and $\tau_n.$  Lemma \ref{lemma:phi_equal} implies that  for any $\alpha:d^n\to d$ and $\beta:d^n \to d^n$ we have $ \varphi_{\alpha \circ \beta} = \alpha \circ (\varphi_{\beta_1},\dots,\varphi_{\beta_n}).$ 
In particular, if we take $\beta={\sf id}_{d^n},$ we get 
\begin{equation}
\varphi_\alpha = \alpha\circ (\varphi_{{\sf pr}_1}, \dots, \varphi_{{\sf pr}_n}).   
\end{equation}
So, if we take $\psi= (\varphi_{{\sf pr}_1},\dots, \varphi_{{\sf pr}_n}),$ we obtain $ \tilde\eta \circ \psi = \varphi.$ Let us prove that such $\psi$ is unique.  Assume that $\psi:e\to d^n$ is a morphism such that $\tilde \eta \circ \psi = \varphi.$ Then 
$\varphi_\alpha = \alpha \circ \psi.$ It follows that
$\psi_i=\varphi_{{\sf pr}_i}.$
\end{proof}

\begin{lemma}\label{lemma:n'}
For  $1\leq n'\leq n,$ if $\varphi:e\to d^{\CC(c,d)}$ is a morphism such that $\sigma_n \varphi=\tau_n\varphi$, then $\sigma_{n'} \varphi=\tau_{n'}\varphi.$ 
\end{lemma}
\begin{proof} It is enough to prove for $n'=n-1\geq 1.$
Take morphisms $\alpha:d^{n-1}\to d$ and $ \beta:c\to d^{n-1}.$ Since $d^n=d^{n-1}\times d$ we have a projection ${\sf pr}_{\leq n-1}:d^n\to d^{n-1}$ and we can take a map $\delta :d^{n-1}\to d^n$  such that ${\sf pr}_{\leq n-1} \circ  \delta={\sf id}_{d^{n-1}}$ and ${\sf pr}_n \circ \delta={\sf pr}_{n-1}$. We set $\alpha'=\alpha \circ {\sf pr}_{\leq n-1}:d^n\to d$ and $\beta'=\delta \circ \beta: c \to d^n.$ Then $\alpha\circ \beta = \alpha'\circ \beta';$ $\beta'_i=\beta_i$ for $1\leq i\leq n-1$ and $\beta'_n=\beta_{n-1}.$ Then by the assumption and Lemma \ref{lemma:phi_equal} we have
\begin{equation}
\varphi_{\alpha\circ \beta}=\varphi_{\alpha'\circ \beta'} = \alpha' \circ (\varphi_{\beta'_1},\dots,\varphi_{\beta'_n})=\alpha \circ (\varphi_{\beta_1},\dots,\varphi_{\beta_{n-1}}).    
\end{equation}
The assertion follows. 
\end{proof}

\begin{remark}
Generally the equation $\sigma_n \varphi = \tau_n \varphi$ for $n\geq 1$ does not imply $\sigma_0 \varphi = \tau_0 \varphi.$ The assumption $n'\geq 1$ of Lemma \ref{lemma:n'} is essential. 
\end{remark}

\begin{proof}[Proof of Theorem \ref{th:powers}]
Lemma \ref{lemma:n'} implies that \begin{equation}
\hom_{{\sf O}^+(d)}^n(\CC(-,d),d)\cong \hom_{{\sf O}^+(d)}^{\leq n}(\CC(-,d),d).
\end{equation}
Let us prove $T_{d^n}\cong \hom_{{\sf O}^+(d)}^{\leq n}(\CC(-,d),d).$ By Lemma \ref{lemma:equalizer} we have
\begin{equation}
\eta: d^n\cong \hom_{{\sf O}^+(d)}^{\leq n}(\CC(d^n,d),d). \end{equation}
So we need to prove the universal property.  Consider a co-augmented functor $F$ such that $d^n\in {\sf Inv}(F).$ Since $d$ is a retract of $d^n,$ we have $d\in {\sf Inv}(F).$ Therefore there is a unique morphism of co-augmented functors $\theta:F\to \hom_{{\sf End}(d)}(\CC(-,d),d).$ Taking the composition with the morphism $\hom_{{\sf End}(d)}(\CC(-,d),d) \to d^{\CC(-,d)},$ we obtain a morphism $\varphi:F\to d^{\CC(-,d)}.$ Then, in order to prove that $T_{d^n}\cong \hom_{{\sf O}^+(d)}^n(\CC(-,d),d)$ it is sufficient to prove that $\sigma_{n'} \varphi_c = \tau_{n'} \varphi_c$ for any $c,$ any $1\leq n'\leq n$ and any morphism of co-augmented functors 
$\varphi:F\to d^{\CC(-,d)}.$
By Lemma \ref{lemma:n'} it is enough to prove that $\sigma_{n} \varphi_c = \tau_{n} \varphi_c.$  

Let us prove that $\sigma_n \varphi_c = \tau_n \varphi_c$ for any $c.$  Take $\alpha:d^n\to d$ and $\beta:c\to d^n.$ Note that $\varphi \eta^F=\tilde \eta.$
The commutative diagram 
\begin{equation}\label{eq:diag1}
\begin{tikzcd}
c\ar[d,"\beta"]\ar[r,"\eta_c^F"] & F(c) \ar[d,"F(\beta)"] \ar[r,"\varphi_c"] & d^{\CC(c,d)} \ar[d,"d^{\CC(\beta,d)}"] \ar[r,"{\sf pr}_{\alpha\circ \beta}"] & d\ar[d,"{\sf id}"] \\
d^n \ar[rrr,bend right=25,"\alpha"] \ar[r,"\eta^F_{d^n}","\cong"'] & F(d^n) \ar[r,"\varphi_{d^n}"] & d^{\CC(d^n,d)} \ar[r,"{\sf pr}_\alpha"] & d
\end{tikzcd}
\end{equation}
shows that 
\begin{equation}
    (\varphi_c)_{\alpha \circ \beta} =  \alpha \circ  (\eta_{d^n}^F)^{-1} \circ  F(\beta).
\end{equation}
And the diagram 
\begin{equation}
\begin{tikzcd}
c\ar[d,"\beta"]\ar[r,"\eta_c^F"] & F(c) \ar[d,"F(\beta)"] \ar[r,"\varphi_c"] & d^{\CC(c,d)} \ar[d,"d^{\CC(\beta,d)}"] \ar[r,"{\sf pr}_{\beta_i}"] & d\ar[d,"{\sf id}"] \\
d^n  \ar[rrr,bend right=25,"{\sf pr}_i"] \ar[r,"\eta^F_{d^n}","\cong"'] & F(d^n) \ar[r,"\varphi_{d^n}"] & d^{\CC(d^n,d)} \ar[r,"{\sf pr}_{{\sf pr}_i}"] & d
\end{tikzcd}
\end{equation}
implies that 
\begin{equation}
    ((\eta_{d^n}^F)^{-1} \circ  F(\beta))_i = (\varphi_c)_{\beta_i}.
\end{equation}
Therefore, we have 
\begin{equation}
(\varphi_c)_{\alpha \circ \beta} = \alpha \circ ((\varphi_c)_{\beta_1},\dots, (\varphi_c)_{\beta_n}).
\end{equation}
Then Lemma \ref{lemma:phi_equal} implies that $\sigma_n \varphi_c=\tau_n \varphi_c.$ This implies that $T_{d^n} = \hom^n_{{\sf O}^+(d)}(\CC(-,d),d).$

Now we prove that $T_{\{1,d^n\}}=\hom^{\leq n}_{{\sf O}(d)}(\CC(-,d),d).$ The proof is similar. We just need to note that if $\eta^F_1:1\to F(1)$ is an isomorphism, then for any natural transformation $\varphi:F\to d^{\CC(-,d)},$ any $\alpha:1 \to d$ and $\beta:c\to 1$ we have a diagram similar to \eqref{eq:diag1}.
\begin{equation}
\begin{tikzcd}
c\ar[d,"\beta"]\ar[r,"\eta_c^F"] & F(c) \ar[d,"F(\beta)"] \ar[r,"\varphi_c"] & d^{\CC(c,d)} \ar[d,"d^{\CC(\beta,d)}"] \ar[r,"{\sf pr}_{\alpha\circ \beta}"] & d\ar[d,"{\sf id}"] \\
1 \ar[rrr,bend right=25,"\alpha"] \ar[r,"\eta^F_1","\cong"'] & F(1) \ar[r,"\varphi_1"] & d^{\CC(1,d)} \ar[r,"{\sf pr}_\alpha"] & d
\end{tikzcd}
\end{equation}
This diagram implies that 
\begin{equation}
(\varphi_c)_{\alpha \circ \beta} = \alpha \circ (),
\end{equation}
where $():F(c)\to 1.$ Then Lemma \ref{lemma:phi_equal} implies that $\sigma_0 \varphi_c=\tau_0 \varphi_c$ and the rest of the proof is the same as for $T_{d^n}.$

The fact that $T_{d^+}=\hom_{{\sf O}^+(d)}(\CC(-,d),d)$ follows from the equations $$T_{d^+}= \varprojlim T_{d^n}$$
and $$\hom_{{\sf O}^+(d)}(\CC(-,d),d)= \varprojlim \hom^{\leq n}_{{\sf O}^+(d)}(\CC(-,d),d).$$

Similarly we have $$T_{d^\bullet}=\varprojlim T_{\{1,d^n\}}=\hom_{{\sf O}(d)}(\CC(-,d),d).$$   
\end{proof}
$$$$

\section{An idempotent pro-completion tower}

 We end with a few comments on a pro-idempotent monad $M_\bullet$ associated with a given monad $M.$ Recall from \cite{Dror73} that the Bousfield-Kan $R$-homology completion tower $R_\bullet X,$ associated with a topological space $X,$ is pro-idempotent.  In addition, and  as consequence, its $R-$ homology $H_*(R_\bullet X,R)$ is naturally pro-isomorphic to the homology $H_*(X,R)$ of $X.$

One would like to have a similar result for a general monad $M.$ This is possible, with the price being the replacement of the 
tot-tower $tot_\bullet X,$ with a slightly 
more involved tower defined inductively. This line was considered to with clear results for a a general co-augmented functor 
in the homotopy category of spaces, by A. Libman, compare \cite{libmanuniversal}.
 
For a general subcategory $\DD\subseteq \CC$ of a nice category $\CC,$ one can construct the right Kan extension functor $T_\DD: \CC\to \CC.$ as above. This functor
is not idempotent. However, we can consider a "refined" right extension 
$$T_\DD^{pro}:\CC\to  pro-\DD.$$

This last functor associate, as usual to each $X\in \CC $ a diagram of objects  in $\DD$  indexed by the coma category $\DD_{X/}.$ The limit, in $\CC,$ of this diagram of objects in $\DD$ is the value of 
 right Kan extension, $T_\DD(X),$ on the object $X,$ see equation \ref{limitovercat} above.

Now  for the diagram $X/\DD\to \DD,$ to define a pro-object
it must be filtering. Therefore if we assume that $D$ is closed under finite limits, i.e. pullbacks. In this case, the diagram  $X/\DD\to \DD$ of objects in $\DD$ is filtering. Hence, the the above define diagram $T_\DD^{pro}$ is a pro-object
in $pro-\CC.$ Moreover, the functor can be directly 
prolonged to a functor:

$$ T^{pro}_\DD :pro-\CC\to pro-\CC$$

which deserves the name "tautological pro-completion" with respect to the inclusion $\DD\subseteq \CC.$  As such it is 
clearly pro-idempotent. For example, if $\DD$ is the full subcategory of {\em Groups} consisting of finite groups, one gets the usual diagram $G\to (\Gamma_i)_i$
of finite all finite groups under a given group $G.$ 

\medskip

Our aim is to show that in case $\DD$ is the closure under finite limits of  the image of a monad
$M,$ there is a small variant on the Bousfield-Kan tower associated with $M,$ which is pro-equivalent to this canonical pro-object $T^{pro}_\DD.$ Moreover, that pro-object is the {\em terminal pro-monad} among those that preserve the closure of $Im D$ under finite limits. Note that the image of $M$ is not, in general, closed under finite limits. Therefore, the following construction, which is valid for any $M,$ will be shown to be equivalent to the above $T^{pro}_{\overline{Im M}}$  where the subscript $\overline {Im M}$ denotes the closure of the image of the monad $M$ under finite limits.

$$$$

 Consider the inductively defined  tower of injective maps  of terminal monads:
 
 $$ M_0=M;  M_{i+1}:= T_{M_i}\to M_i\cdots \to M. $$

\medskip

To continue, we notice that under no additional assumptions on $\CC$, one has associated with 
$M$ an idempotent  pro-monad tower, $(M_i)_{i<\omega}.$ The limit of this tower
of  monads is precisely the terminal monad 
that preserves the closure of $Im M,$ the image of our monad $M,$ under all {\em finite} limits.
$$$$

By   \ref{preservelimit}, each $M_i$
in this tower not only preserves $M_j$
for $j<i,$ but also any finite limit of 
objects of the form $M_j(Y_j)$ for $Y_j\in \CC.$ 
It follows that if $X\in \CC$ is in the closure under finite limits of the image of $M,$ then for $i$ large enough, there is an equivalence $X\cong M_i(X).$ Since each $M_i (X)$ is, by construction, an element in the above finite limits closure of 
 $Im M$ in $\CC.$ Therefore we have a pro-idempotent tower:

$$
M_\bullet \cong  M_\bullet\circ  M_\bullet.
$$

In addition, using the argument in \cite{BousfieldKan1973completion} and \cite{CasacubertaFrei} it follows that for every monad $M$ one has a pro-equivalence:

$$ M\cong  M(M_\bullet).$$

In other words, the tower $M_\bullet$ satisfies  some of the basic properties  of the 
classical Bousfield-Kan  $R$-homology
completion tower as given by \cite{Dror73}.

The limit, $M_\infty=lim_iM_i,$ of the tower $M_i$
is the terminal monad among all co-augmented functors that preserve $\overline{Im M},$
the closure of the image of $M$ under finite limits.

$$$$

\begin{remark}
    
 following Fakir, \cite{Fakiridempotentmonads} one can continue to this tower of inclusions,  see \ref{equalizerequaterminal} above, trasfinitely. Under suitable  rather weak assumption on $\CC$ this tower converges to an idempotent monad $L_M.$
This idempotent monad is easily seen to be the terminal monad $T_{\overline{\DD}}$ where now $\overline{\DD}$ denotes the closure of the image of $M$ under {\em all limits.} In the infinity category of spaces
the classical example is the map $L_{HR}\to R_{\infty}$ from the idempotent Bousfield 
homological localization to the $R-$ completion functor on spaces, compare \cite{CasacubertaFrei}.
\end{remark}

{\bf Examples:}
In the category of groups we can
consider the terminal monad $T_G$ associated with a group $G,$ so that $T_G(G)\cong G.$
It is given as above by the double dual

$$T_G(\Gamma)\cong map_{End\,G}(map(\Gamma,G),G).$$ 

This is a subgroup of $G^l$ for $l=|map(\Gamma,G)|,$ the cardinality of the set of homomorphisms.
 For the group of integers $ \Gamma=\zz,$ we have $T_G(\zz)\subseteq G^{|G|}.$ The transfinite tower  of Fakir stabilizes at
$\zz\to L_G(\zz)\cong C,$ where $C$ is a cyclic group whose order is the   LCM of the orders of all elements of $G,$ since the image of the generator of $\zz\to T_G(\zz,)$ is the diagonal element $(g_g)_{g\in G}.$
Namely,  $L_G(\zz)$ is the image  subgroup of $\zz\to T_G(\zz).$  This  $L_G$ localization, or reflexive functor can be characterized as terminal among those that preserve the closure in the category of groups of $\{G\}$ under all limits, mapping $\CC$ into that closure, or the initial idempotent monad that turns every $map(-,G)$-equivalence (i.e. a sort of 'G-cohomology equivalence') into an equivalence.
\break

  {\em An infinity categorical example} see \cite{Regimovmonads}. The classical $R_\infty$
of Bousfield and Kan comes  from the monad
$R$ on the $\infty$-category of topological spaces.
It 
preserves not only spaces of the form $RX,$
i.e. $R-$GEMs but also $R$-polyGEM spaces i.e.  the closure of  $R-$GEMs under finite limits. In this special case the construction of the terminal $R_\infty$
is somewhat simpler than the above inductive
tower $(M_i)_i.$
\\
The  precise meaning or value of this tower for the (discrete)
pro-finite completion of groups, considered as a monad $\MM(G)=\widehat G,$   is not immediately clear. For every group $G,$ the monad
$\MM_\infty G$ is a natural subgroup of the pro-finite completion $\widehat G$ of $G;$ with the property that it is idempotent  ($\MM_\infty \MM_\infty (G)\cong \MM_\infty(G))$ if
$G$ is a finite limit of   (discrete) profinite groups.  In fact, since $\MM_\infty$ preserves all finite limits in  $Im \MM,$ see above, we have that it preserves  any  finite limit of profinite groups.  Note, however, that for a finitely presented group $G,$ one has an isomorphism, $\MM_\infty G=\widehat G,$ since for such a group one has an isomorphism  $\widehat G\cong \widehat{\widehat G},$ namely, the completion  $\MM=\widehat {(-)}$ is   idempotent on this subcategory of groups. The transfinite intersection or  limit of $lim_\alpha \MM_\alpha$ on all ordinals is also  an interesting subgroup of the pro-finite completion, which is just  $\widehat G$ itself if $G$
is finitely presented.

\printbibliography
\end{document}